\documentclass[a4paper]{elsarticle}
\usepackage{amssymb,amsmath,amsthm}
\usepackage[T1]{fontenc}
\usepackage[utf8]{inputenc}
\usepackage{fancyvrb}
\usepackage{enumitem}
\usepackage{longtable}
\usepackage{color}
\definecolor{darkblue}{RGB}{42,31,109}
\definecolor{darkgreen}{RGB}{0,100,0}
\definecolor{darkgray}{RGB}{100,110,120}
\usepackage{hyperref}
\hypersetup{unicode=true,
            pdfborder={0 0 0},
            colorlinks=true,
            urlcolor=darkblue,
            linkcolor=black,
            citecolor=darkgreen,
            breaklinks=true}
\usepackage{ifthen}
\usepackage{lscape}
\usepackage{tikz}
\usepackage{pgfmath}
\usetikzlibrary{matrix,positioning}
\usetikzlibrary{arrows,automata}
\DeclareMathOperator{\rev}{\mathrm{rev}}

\newcommand{\word}{w}
\newcommand{\A}{\texttt{a}}
\newcommand{\D}{\texttt{d}}

\theoremstyle{plain}
\newtheorem{theorem}{Theorem}[section]
\newtheorem{lemma}[theorem]{Lemma}
\newtheorem{proposition}[theorem]{Proposition}
\newtheorem*{proposition*}{Proposition}

\newtheorem*{conjecture*}{Conjecture}

\theoremstyle{definition}
\newtheorem{definition}[theorem]{Definition}

\theoremstyle{remark}

\newtheorem*{remark*}{Remark}

\usepackage{array}

\def\customwidth{0.6pt}
\def\customscale{0.61}
\def\customwidthsmall{0.6pt}
\def\customscalesmall{0.5}
\newcommand{\comparray}[2]{
  \begin{scope}[shift={#1}]
    \foreach \row [count=\y from 0] in {#2} {
      \pgfmathparse{dim(\row)};
      \pgfmathsetmacro{\w}{\pgfmathresult};
      \foreach \b [count=\x from 0] in \row {
        \ifthenelse{\b=1}
          {\draw[line cap=round] (\x,-\y) -- (\x,1-\y);}{};
        \ifthenelse{\b=2}
          {\draw[loosely dotted,line cap=round] (\x,-\y) -- (\x,1-\y);}{};
        \ifthenelse{\b=3}
          {\draw[red,line cap=round] (\x,-\y) -- (\x,1-\y);}{};
      }
      \draw[line cap=round] (0,1-\y-1) -- (\w-1,-\y);
      \ifthenelse{\y=0}
        {\draw[line cap=round] (0,1-\y) -- (\w-1,1-\y);}
        {};
  }
  \end{scope}
}
\newcommand{\permutation}[3]{
  \begin{scope}[shift={#1}]
    \foreach \v [count=\x from 0] in {#3} {
      \node[text height=1.5ex,text depth=.25ex] at (\x+0.5,-#2+0.5) {$\v$};
  }
  \end{scope}
}

\title{Enumerating permutations sortable by \\$k$ passes through a pop-stack}
\author{Anders Claesson\corref{cor1}\fnref{fn1}}
\author{Bjarki Ágúst Guðmundsson\fnref{fn1}}
\cortext[cor1]{Corresponding author, {\tt akc@hi.is}}
\fntext[fn1]{This work was supported
  by The Doctoral Grants of the University of Iceland Research Fund and
  The Icelandic Infrastructure Fund}
\address{Science Institute, University of Iceland, Dunhaga 5, IS-107 Reykjavik, Iceland}

\begin{document}

\begin{abstract}
  In an exercise in the first volume of his famous series of books,
  Knuth considered sorting permutations by passing them through a stack.
  Many variations of this exercise have since been considered, including
  allowing multiple passes through the stack and using different data
  structures. We are concerned with a variation using pop-stacks that was
  introduced by Avis and Newborn in 1981. Let $P_k(x)$ be the generating
  function for the permutations sortable by $k$ passes through a
  pop-stack.  The generating function $P_2(x)$ was recently given by
  Pudwell and Smith (the case $k=1$ being trivial). We show that
  $P_k(x)$ is rational for any $k$. Moreover, we give an algorithm to
  derive $P_k(x)$, and using it we determine the generating functions
  $P_k(x)$ for $k\leq 6$.
\end{abstract}
\begin{keyword}
  enumeration\sep generating function\sep pop-stack\sep permutation\sep rational\sep automaton
\end{keyword}

\maketitle

\thispagestyle{empty}

\section{Introduction}

Knuth~\cite[Exercise~2.2.1.5]{knuth1968art} noted that permutations
sortable by a stack are precisely those that do not contain a
subsequence in the same relative order as the permutation $231$.  This
exercise inspired a wide range of research and can be seen as the
starting point of the research field we now call permutation patterns.
Our interest lies in Knuth's original exercise and its variations.

In 1972 Tarjan~\cite{tarjan1972sorting} considered sorting with networks
of stacks and queues, a problem that, in general, has proven itself to
be beyond the reach of current methods in permutation patterns and
enumeration.  There has been some recent significant progress though: In
2015 Albert and Bousquet-M{\'e}lou~\cite{albert2015permutations}
enumerated permutations sortable by two stacks in parallel. Sorting with
two stacks in series is, however, an open problem. A related problem is
that of sorting permutations by $k$ passes through a stack, where the
elements on the stack are required to be increasing when read from top
to bottom. West~\cite{west1990permutations} characterized the
permutations sortable by two passes through a stack in terms of pattern
avoidance and conjectured their enumeration, a conjecture that was
subsequently proved by
Zeilberger~\cite{zeilberger1992proof}. Permutations sortable by three
passes have been characterized by
\'Ulfarsson~\cite{ulfarsson2012describing}, but their enumeration is
unknown.

In other variations of Knuth's exercise different data structures are
used for sorting. One notable example is that of pop-stacks: a stack
where each pop operation completely empties the stack. Avis and
Newborn~\cite{avis1981pop} enumerated the permutations sortable by
pop-stacks in series, with the modification that each pop leads to the
popped elements being immediately pushed to the following
pop-stack. Atkinson and Stitt~\cite{atkinson2002restricted} considered
two pop-stacks in genuine series. Permutations sortable by pop-stacks in
parallel have been studied by Atkinson and Sack~\cite{atkinson1999pop},
who characterized those permutations by a finite set of forbidden
patterns. They also conjectured that the generating function for their
enumeration is rational, which was subsequently proved by Smith and
Vatter~\cite{smith2009enumeration}, who gave an insertion
encoding~\cite{albert2005insertion} for the sortable permutations.

Pudwell and Smith~\cite{pudwell2017} recently
characterized the permutations sortable by two passes through a pop-stack in
terms of pattern avoidance and gave their enumeration. They also gave a
bijection between certain families of polyominoes and the permutations sortable
by one or two passes through a pop-stack, but noted that the bijection does not
generalize to three passes.
In this paper we consider, more generally, the permutations sortable by
$k$ passes through a pop-stack. In particular, we give an algorithm to
derive a rational generating function for the permutations sortable
by $k$ passes through a pop-stack, for any fixed $k$.

\section{Sorting plans and traces}

A single pass of the pop-stack sorting operator formally works as
follows. Processing a permutation $\pi=a_1a_2\dots a_n$ of
$[n]=\{1,\dots,n\}$ from left to right, if the stack is empty or its top
element is smaller than the current element $a_i$ then perform a single
pop operation ($\A$), emptying the stack and appending those elements to
the output permutation; else do nothing ($\D$). Next, push $a_i$ onto the
stack and proceed with element $a_{i+1}$, or if $i=n$ perform one final
pop operation ($\A$), again emptying the stack onto the output
permutation, and terminate. Define $P(\pi)$ as the final output
permutation and $\word(\pi)$ as the word over the alphabet $\{\A,\D\}$
defined by the operations performed when processing $\pi$.  For
instance, with $\pi=752491863$ we have $P(\pi)=257419368$ and
$\word(\pi)=\A\D\D\A\A\D\A\D\D\A$. Note that $\word(\pi)$ will always
begin and end with the letter $\A$. We will call any word of length $n+1$
with letters in $\{\A,\D\}$ that begin and end with the letter $\A$ an
\emph{operation sequence}.

Let us now introduce what we call sorting traces.  Consider applying the
pop-stack operator $P$ to a permutation $\pi$ of $[n]$. Start by interleaving $w(\pi)$
with $\pi$; for instance, with $\pi=752491863$ (as before) we have
$\A 7\D 5\D 2\A 4\A 9\D 1\A 8\D 6\D 3 \A$.  Replacing $\A$ with a bar
and $\D$ with a space, and placing $P(\pi)$ below this string, we have
$$
\begin{tikzpicture}[line width=\customwidth, scale=\customscale]
  \useasboundingbox (0,-0.5) rectangle (9,1);
  \comparray{(0.0,0)}{{1,0,0,1,1,0,1,0,0,1}};
  \permutation{(0.0,0)}{0}{7,5,2,4,9,1,8,6,3};
  \permutation{(0.0,0)}{1}{2,5,7,4,1,9,3,6,8};
\end{tikzpicture}
$$
We call the numbers between pairs of successive \A's \emph{blocks}. Above,
the blocks of $\pi$ are $752$, $4$, $91$, and $863$.
Note that $P(\pi)$ can be obtained from $\pi$ by reversing each block.
An index $i\in[n-1]$ of a permutation $\pi=a_1a_2\dots a_n$ is an
\emph{ascent} if $a_i<a_{i+1}$.
Similarly, $i$ is a \emph{descent} if $a_i>a_{i+1}$. With this terminology
 $\word(\pi) = c_1c_2\dots c_{n+1}$ is simply the ascent/descent word of
$-\infty\hspace{1pt}\pi\hspace{1pt}\infty$; i.e.\ $c_1=c_{n+1}=\A$ and, for $2\leq i\leq n$,
$$c_i =
\begin{cases}
  \A &\text{if $i-1$ is an ascent,} \\
  \D &\text{if $i-1$ is a descent.}
\end{cases}
$$

A figure as the one above can be
extended to depict multiple passes through a pop-stack. Applying
the pop-stack operator to the example permutation, $\pi$, until it is sorted gives:
$$
\begin{tikzpicture}[line width=\customwidth, scale=\customscale]
  \useasboundingbox (0,-3.5) rectangle (9,1);
  \comparray{(0.0,0)}{{1,0,0,1,1,0,1,0,0,1},
    {1,1,1,0,0,1,0,1,1,1},
    {1,1,0,1,1,0,1,0,1,1},
    {1,0,1,0,0,1,0,1,0,1}};
  \permutation{(0.0,0)}{0}{7,5,2,4,9,1,8,6,3};
  \permutation{(0.0,0)}{1}{2,5,7,4,1,9,3,6,8};
  \permutation{(0.0,0)}{2}{2,5,1,4,7,3,9,6,8};
  \permutation{(0.0,0)}{3}{2,1,5,4,3,7,6,9,8};
  \permutation{(0.0,0)}{4}{1,2,3,4,5,6,7,8,9};
\end{tikzpicture}
$$
We will call such figures sorting traces, or traces for short.
The structure that remains when removing the numbers from a trace we
call its sorting plan.  Each row of a sorting plan corresponds to an operation
sequence, and for convenience we shall number the rows $1$ through $k$,
from top to bottom. The example sorting plan can be viewed as the following
array of operation sequences:
\[
\begin{array}{c}
  \A\,\D\,\D\,\A\,\A\,\D\,\A\,\D\,\D\,\A \\
  \A\,\A\,\A\,\D\,\D\,\A\,\D\,\A\,\A\,\A \\
  \A\,\A\,\D\,\A\,\A\,\D\,\A\,\D\,\A\,\A \\
  \A\,\D\,\A\,\D\,\D\,\A\,\D\,\A\,\D\,\A
\end{array}
\]
By interpreting each column as a binary number with $a=0$ and $d=1$ the
sorting plan can more compactly be represented, or \emph{encoded}, with
the sequence of numbers
\begin{center}
0, 9, 10, 5, 5, 10, 5, 10, 9, 0.
\end{center}

Formally, we define a trace and its sorting plan as follows.

\begin{definition}
    \label{def:trace}
    Let $A = (\alpha_1,\alpha_2,\ldots,\alpha_k,\alpha_{k+1})$ be a
    $(k+1)$-tuple of permutations of $[n]$ in which $\alpha_{k+1}$ is the
    identity permutation. Let
    $M = (\mu_1,\mu_2,\ldots,\mu_k)$ be a $k$-tuple of
    operation sequences, each of length $n+1$.
    We call $T=(A,M)$ a \emph{trace} of \emph{length} $n$ and
    \emph{order} $k$, and $M$ its \emph{sorting plan}, if the following
    conditions are satisfied for $i=1,\ldots,k$:
    \begin{enumerate}
    \item\label{traceprop1} The word $\mu_i$ records the sequence of
      operations performed by the pop-stack operator when applied to
      $\alpha_i$---in symbols, $\word(\alpha_i)=\mu_i$. As to the
      picture of the trace, two adjacent numbers form an ascent if and
      only if they are separated by a bar.
    \item\label{traceprop2} The sequence of permutations in $A$ records
      repeated application of the pop-stack operator to the permutation
      $\alpha_1$---in symbols, $P(\alpha_i) = \alpha_{i+1}$.  Or,
      assuming the first condition is satisfied,
      $\rev(\alpha_i,\mu_i) = \alpha_{i+1}$, where the operator $\rev$
      is defined below.
    \end{enumerate}
\end{definition}

\begin{definition}
  \label{def:rev}
  Let $\pi$ be a permutation of $[n]$ and $\mu=c_1c_2\dots c_{n+1}$ be an
  operation sequence. Let
  $i_1< i_2< \dots< i_k$ be the sequence of indices $i$ for which
  $c_i=\A$.  Write $\pi=\gamma_1\gamma_2\dots\gamma_{k-1}$, where
  $|\gamma_j|=i_{j+1}-i_{j}$ for $j=1,\dots,k-1$. In other words, the
  length of $\gamma_j$ is the same as the length of the $j$th block of
  $\pi$ with respect to $\mu$.  Define
  $\rev(\pi,\mu) = \gamma_1^r\gamma_2^r\dots\gamma_{k-1}^r$, where $(\,\cdot\,)^r$ is the reversal operator.
  We call this the \emph{blockwise reversal of $\pi$ according to $\mu$}.
\end{definition}

Because of the strict constraints on the operation sequences
and the permutations, much of the information stored in a trace is redundant.
As we have seen, the first permutation $\alpha_1$ alone determines the complete
trace. Similarly, if we have the sorting plan $M$, then
we can recover the permutations $\alpha_1$, \dots, $\alpha_{k+1}$. This is
possible because the last permutation $\alpha_{k+1}$ is the identity,
the permutation $\alpha_k$ is the blockwise reversal of $\alpha_{k+1}$ according to $\mu_k$,
the permutation $\alpha_{k-1}$ is the blockwise reversal of $\alpha_{k}$ according to $\mu_{k-1}$, etc.
In symbols, $\rev(\alpha_i,\mu_i) = \alpha_{i+1}$ if and only if
$\alpha_i = \rev(\alpha_{i+1},\mu_i)$. In this way a sorting plan uniquely
determines a trace.

Our goal is to count how many permutations of $[n]$ are sortable by $k$
passes through a pop-stack. Note that a permutation sortable by $k$
passes is also sortable by $k+1$ passes. For brevity we will sometimes
refer to a permutation sortable by $k$ passes through a pop-stack as
\emph{$k$-pop-stack-sortable}. Starting with a $k$-pop-stack-sortable
permutation of $[n]$ and performing $k$ passes of the pop-stack sorting
operator results in a trace of length $n$ and order $k$. Conversely, the
first row of that trace is the $k$-pop-stack-sortable permutation we
started with. Thus, $k$-pop-stack-sortable permutations of $[n]$ are in
one-to-one correspondence with traces of length $n$ and order $k$. Those
are, in turn, in one-to-one correspondence with their sorting plans, and
hence it suffices to count sorting plans of length $n$ and order $k$.

Let us call any $k$-tuple of operation sequences of length $n+1$ an
\emph{operation array} of length $n+1$ and order $k$. Note that
sorting plans are operation arrays, but not all operation arrays are sorting plans.
Let an operation array $(\mu_1,\mu_2,\ldots,\mu_k)$ be given. Let the
permutations $\alpha_1$, $\alpha_2$, \ldots, $\alpha_{k+1}$ be defined
by requiring that the permutation $\alpha_{k+1}$ is the identity,
the permutation $\alpha_k$ is the blockwise reversal of $\alpha_{k+1}$ according to $\mu_k$,
the permutation $\alpha_{k-1}$ is the blockwise reversal of $\alpha_{k}$ according to $\mu_{k-1}$,
etc. In symbols, $\rev(\alpha_i,\mu_i) = \alpha_{i+1}$. Then
the tuple $(\alpha_1,\ldots,\alpha_{k+1})$ is called a
\emph{semitrace}. In other words, a semitrace is defined by requiring
that Property~\ref{traceprop2} of Definition~\ref{def:trace} holds, but
not necessarily Property~\ref{traceprop1} of the same definition.

We shall characterize those operation arrays
that are sorting plans, then count the sorting plans of order $k$, and by
extension the $k$-pop-stack-sortable permutations. We start by
considering cases where $k$ is small.

\subsection{1-pop-stack-sortable permutations}

Let us count permutations of $[n]$ sortable by one pass through a
pop-stack, or, equivalently, sorting plans of length $n$ and order $1$. We
want to determine which operation sequences of length $n+1$ when
interpreted as operation arrays are sorting plans.
Consider the operation sequence $\A\D\D\A\A\D\A\D\A$ as an
example. The semitrace corresponding to it is
$$
\begin{tikzpicture}[line width=\customwidth, scale=\customscale]
  \useasboundingbox (0,-0.5) rectangle (8,1);
  \comparray{(0.0,0)}{{1,0,0,1,1,0,1,0,1}};
  \permutation{(0.0,0)}{0}{3,2,1,4,6,5,8,7};
  \permutation{(0.0,0)}{1}{1,2,3,4,5,6,7,8};
\end{tikzpicture}
$$
Is this a trace? By definition of a semitrace it satisfies
Property~\ref{traceprop2} of Definition~\ref{def:trace}, but does it
also satisfy Property~\ref{traceprop1}? Yes, because the bottom-most
permutation is the identity we find that any two adjacent numbers
separated by a bar form an ascent, and two adjacent numbers within a
block (i.e.\ adjacent numbers that are not separated by a bar) form a
descent.
Thus, each operation sequence represents a sorting plan of order 1, and
hence a 1-pop-stack-sortable permutation.  An operation sequence starts
and ends with the letter $\A$. The remaining letters can be either $\A$
or $\D$. Thus, for $n>0$, there are $2^{n-1}$ operation sequences of
length $n+1$, and hence $2^{n-1}$ 1-pop-stack-sortable permutations of
$[n]$. Based on how they are derived from the operation sequences it is
easy to see that these are precisely the layered permutations
(direct sums of decreasing permutations).

While this simple example outlines our approach to count
$k$-pop-stack-sortable permutations, it is a little too simple. For
larger $k$, most operation arrays will not be sorting plans. Before
attacking the problem in full generality, let us consider the case $k=2$
and see how to restrict the operation arrays to those that
represent sorting plans.

\subsection{2-pop-stack-sortable permutations}

Pudwell and Smith \cite{pudwell2017} enumerated 2-pop-stack-sortable
permutations. Let us reproduce their results by classifying sorting
plans of order $2$.  Given an arbitrary operation array of order 2,
i.e.\ a pair of operation sequences of length $n+1$, we will start with
the identity permutation at the bottom, fill in the remaining two
permutations, and then determine if the resulting semitrace is a trace.  As
we saw in the previous section, the second operation sequence will
always satisfy Property~\ref{traceprop1}. To determine when the first
operation sequence satisfies Property~\ref{traceprop1} let us do case
analysis based on the size of an arbitrary block in the second row of
our trace.

If the block is of size $1$, then we have a single integer
$a\in[n]$ and the neighborhood around the block looks as follows:
$$
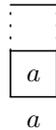
\begin{tikzpicture}[line width=\customwidth, scale=\customscale]
  \useasboundingbox (0,-1.5) rectangle (1,1);
  \comparray{(0.0,0)}{{2,2},{1,1}};
  \permutation{(0.0,0)}{1}{a};
  \permutation{(0.0,0)}{2}{a};
\end{tikzpicture}
$$
Here, a dotted line can either be solid or not, representing the
presence or absence of a bar, respectively.

Since at least two numbers are needed to break
Property~\ref{traceprop1} all four configurations of the two
dotted lines are possible:
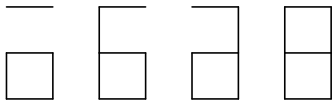
\[
  \begin{tikzpicture}[line width=\customwidth, scale=\customscale]
    \useasboundingbox (0,-1+0.2) rectangle (8,1);
    \comparray{(6.5,0)}{{1,1},{1,1}};
    \comparray{(4.5,0)}{{0,1},{1,1}};
    \comparray{(2.5,0)}{{1,0},{1,1}};
    \comparray{(0.5,0)}{{0,0},{1,1}};
  \end{tikzpicture}
\]

If the block is of size $2$, then we have two integers
$a,b\in[n]$, with $b=a+1$, and the neighborhood around the block looks
as follows:
$$
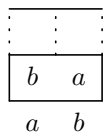
\begin{tikzpicture}[line width=\customwidth, scale=\customscale]
  \useasboundingbox (0,-1.5) rectangle (2,1);
  \comparray{(0.0,0)}{{2,2,2},{1,0,1}};
  \permutation{(0.0,0)}{1}{b,a};
  \permutation{(0.0,0)}{2}{a,b};
\end{tikzpicture}
$$
Consider the dotted line in the middle and assume for a moment
that it is not solid. Then $a$ and $b$
will appear together in the block above, with $a$ appearing before $b$,
and this means that the numbers within this block are not in decreasing
order, a contradiction. Hence the dotted line in the middle must be
solid.
Now, assume that the remaining two dotted lines are solid
too. Then the placement of $a$ and $b$ in the first
row is determined:
$$
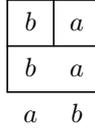
\begin{tikzpicture}[line width=\customwidth, scale=\customscale]
  \useasboundingbox (0,-1.5) rectangle (2,1);
  \comparray{(0.0,0)}{{1,1,1},{1,0,1}};
  \permutation{(0.0,0)}{0}{b,a};
  \permutation{(0.0,0)}{1}{b,a};
  \permutation{(0.0,0)}{2}{a,b};
\end{tikzpicture}
$$
But here we reach a contradiction: in the first
row, $a$ and $b$ form a descent but are separated by a bar.
Thus, at most one of the two dotted lines
on the boundary can be solid. This leaves three possibilities, and we
will not be able to reduce their number any further:
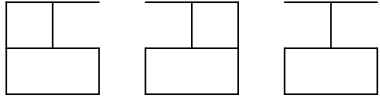
\[
  \begin{tikzpicture}[line width=\customwidth, scale=\customscale]
    \useasboundingbox (0,-4+0.2) rectangle (8,-2);
    \comparray{(0,-3)}{{1,1,0},{1,0,1}};
    \comparray{(3,-3)}{{0,1,1},{1,0,1}};
    \comparray{(6,-3)}{{0,1,0},{1,0,1}};
    \end{tikzpicture}
  \]

If the block is of size $3$, then we have three integers
$a,b,c\in[n]$, with $b=a+1$ and $c=b+1$, and the neighborhood around
the block looks as follows:
$$
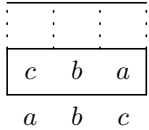
\begin{tikzpicture}[line width=\customwidth, scale=\customscale]
  \useasboundingbox (0,-1.5) rectangle (3,1);
  \comparray{(0.0,0)}{{2,2,2,2},{1,0,0,1}};
  \permutation{(0.0,0)}{1}{c,b,a};
  \permutation{(0.0,0)}{2}{a,b,c};
\end{tikzpicture}
$$
If either $a$ and $b$, or $b$ and $c$ were together in a block in the
first row, then the numbers in that block would not be in decreasing
order, a contradiction. The two centermost dotted lines
must thus be solid:
$$
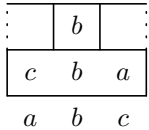
\begin{tikzpicture}[line width=\customwidth, scale=\customscale]
  \useasboundingbox (0,-1.5) rectangle (3,1);
  \comparray{(0.0,0)}{{2,1,1,2},{1,0,0,1}};
  \permutation{(0.0,0)}{0}{,b,};
  \permutation{(0.0,0)}{1}{c,b,a};
  \permutation{(0.0,0)}{2}{a,b,c};
\end{tikzpicture}
$$
If either of the two remaining dotted lines were solid, then either $a$ and
$b$, or $b$ and $c$ would form an ascent and be separated by a bar, a
contradiction. Hence neither of the two remaining dotted lines can be solid,
leaving only one possibility:
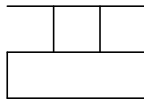
\[
  \begin{tikzpicture}[line width=\customwidth, scale=\customscale]
    \useasboundingbox (2.5,-7+0.2) rectangle (5.5,-5);
    \comparray{(2.5,-6)}{{0,1,1,0},{1,0,0,1}};
  \end{tikzpicture}
\]

For blocks of size $4$ or greater, we present the following lemma.
\begin{lemma}
    \label{lem:blocks3}
  In a trace of order $2$ or greater, each operation sequence---except for
  the first one---contains at most $2$ consecutive $\D$'s. Or, equivalently,
  each row of the sorting plan---except for the first one---has blocks of size at
  most $3$.
\end{lemma}
\begin{proof}
  Given a trace of order $2$ or greater, consider any
  block from any row of its sorting plan, excluding the first row. Assume the block
  has size $k\geq 4$, and that it contains the numbers $a_1, a_2, \ldots, a_k$,
  where $a_1 > a_2 > \cdots > a_k$.
  The neighborhood around the block then looks as follows:
  $$
  
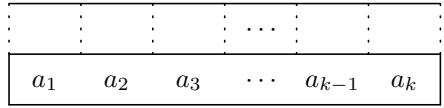
\begin{tikzpicture}[line width=\customwidth, yscale=\customscale*1.1, xscale=\customscale*1.55]
    \useasboundingbox (0,-1+0.2) rectangle (6,1);
    \comparray{(0.0,0)}{{2,2,2,2,2,2,2},{1,0,0,0,0,0,1}};
    \permutation{(0.0,0)}{0}{,,,\hspace{0.2em}\cdots,,};
    \permutation{(0.0,0)}{1}{a_1,a_2,a_3,\hspace{0.2em}\cdots,a_{k-1},a_k};
  \end{tikzpicture}
  $$
  Consider the dotted lines in the upper row, except for the
  two outermost dotted lines. If any of them were not solid, then at least two
  of $a_1,a_2,\ldots,a_k$ would be together in a block. These two
  elements would occur in increasing order, violating Property~\ref{traceprop1}
  of Definition~\ref{def:trace}. These dotted lines must thus be solid:
  $$
  
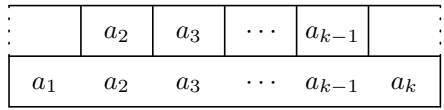
\begin{tikzpicture}[line width=\customwidth, yscale=\customscale*1.1, xscale=\customscale*1.55]
    \useasboundingbox (0,-1+0.2) rectangle (6,1);
    \comparray{(0.0,0)}{{2,1,1,1,1,1,2},{1,0,0,0,0,0,1}};
    \permutation{(0.0,0)}{0}{,a_2,a_3,\hspace{0.2em}\cdots,a_{k-1},};
    \permutation{(0.0,0)}{1}{a_1,a_2,a_3,\hspace{0.2em}\cdots,a_{k-1},a_k};
  \end{tikzpicture}
  $$
  Now, however, $a_2$ and $a_3$ are two adjacent numbers in distinct
  blocks that do not form an ascent, again violating
  Property~\ref{traceprop1}. Hence our assumption that $k\geq 4$ must
  have been false.
\end{proof}

Returning to the traces of order $2$, this lemma tells us that our case
analysis, above, is complete. We have thus restricted the possible local
neighborhoods, based on the size of any block in the second row, to the
following:
$$
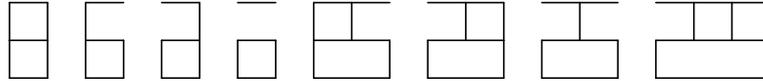
\begin{tikzpicture}[line width=\customwidthsmall, scale=\customscalesmall]
  \useasboundingbox (0.5,-1+0.1) rectangle (20.5,1);
  \comparray{(0.5,0)}{{1,1},{1,1}};
  \comparray{(2.5,0)}{{1,0},{1,1}};
  \comparray{(4.5,0)}{{0,1},{1,1}};
  \comparray{(6.5,0)}{{0,0},{1,1}};
  \comparray{(8.5,0)}{{1,1,0},{1,0,1}};
  \comparray{(11.5,0)}{{0,1,1},{1,0,1}};
  \comparray{(14.5,0)}{{0,1,0},{1,0,1}};
  \comparray{(17.5,0)}{{0,1,1,0},{1,0,0,1}};
\end{tikzpicture}
$$
We have shown that these eight possibilities are necessary, in the sense
that, if we take a sorting plan of order $2$ and slice it up along the
boundaries of the blocks in the second row, then the pieces must
all be contained in the above set. What is perhaps a bit surprising is
that this holds in the other direction as well: gluing together any
proper sequence of the above pieces gives a sorting plan, where proper
means that
\begin{itemize}
    \item the first piece starts with a fully solid boundary,
    \item the last piece ends with a fully solid boundary, and
    \item for every pair of adjacent pieces, the right boundary of the left
        piece and the left boundary of the right piece coincide.
\end{itemize}
This follows from the more general results of the next section. We will not
provide a proof for this special case here, but the intuition is that
all scenarios that could violate Property~\ref{traceprop1} happen
locally, and the analysis above looks at large enough neighborhoods to
get rid of all of these bad scenarios.

With this classification it now becomes straightforward to count the
sorting plans of order $2$. Let $C$ be the ordinary generating function for
``closed'' sequences; that is, proper sequences of pieces that both begin
and end with a fully solid boundary.  Let $H$ be the ordinary generating
function for ``half open'' sequences; that is, proper sequences that begin
with a fully solid boundary but end with a half-solid boundary. Then
\def\notverybig{0.23}
\begin{align*}
    C &\,=\,
      \begin{tikzpicture}[semithick, scale=\notverybig, baseline={([yshift=-.6ex]current bounding box.center)}]
        \comparray{(0,0)}{{1},{1}}
      \end{tikzpicture}
      \,+\,
      C\hspace{1pt}
      \begin{tikzpicture}[semithick, scale=\notverybig, baseline={([yshift=-.6ex]current bounding box.center)}]
        \comparray{(0,0)}{{1,1},{1,1}}
      \end{tikzpicture}
      \,+\,
      H\!\left(\,
      \begin{tikzpicture}[semithick, scale=\notverybig, baseline={([yshift=-.6ex]current bounding box.center)}]
        \comparray{(0,0)}{{0,1},{1,1}}
      \end{tikzpicture}
      \,+\,
      \begin{tikzpicture}[semithick, scale=\notverybig, baseline={([yshift=-.6ex]current bounding box.center)}]
        \comparray{(0,0)}{{0,1,1},{1,0,1}}
      \end{tikzpicture}
      \,\right);
      \\[0.9ex]
    H &\,=\,
      C\!\left(\,
      \begin{tikzpicture}[semithick, scale=\notverybig, baseline={([yshift=-.6ex]current bounding box.center)}]
        \comparray{(0,0)}{{1,0},{1,1}}
      \end{tikzpicture}
      \,+\,
      \begin{tikzpicture}[semithick, scale=\notverybig, baseline={([yshift=-.6ex]current bounding box.center)}]
        \comparray{(0,0)}{{1,1,0},{1,0,1}}
      \end{tikzpicture}
      \,\right)
      \,+\,
      H\!\left(\,
      \begin{tikzpicture}[semithick, scale=\notverybig, baseline={([yshift=-.6ex]current bounding box.center)}]
        \comparray{(0,0)}{{0,0},{1,1}}
      \end{tikzpicture}
      \,+\,
      \begin{tikzpicture}[semithick, scale=\notverybig, baseline={([yshift=-.6ex]current bounding box.center)}]
        \comparray{(0,0)}{{0,1,0},{1,0,1}}
      \end{tikzpicture}
      \,+\,
      \begin{tikzpicture}[semithick, scale=\notverybig, baseline={([yshift=-.6ex]current bounding box.center)}]
        \comparray{(0,0)}{{0,1,1,0},{1,0,0,1}}
      \end{tikzpicture}
      \,\right).
    \intertext{Using the formal variable $x$ to keep track of the length of the partial
        sorting plan we get}
    C &\,=\, 1 + xC + (x + x^2)H; \\
    H &\,=\, (x + x^2)C + (x + x^2+ x^3)H.
\end{align*}
The sequences that both begin and end with a fully solid boundary are
the sorting plans, so the ordinary generating function for sorting plans of
order $2$, and hence 2-pop-stack-sortable permutations, is $C$.
Solving for it we recover Pudwell and Smith's~\cite{pudwell2017}
result:
$$
C = (x^3 + x^2 + x - 1)/(2x^3 + x^2 + 2x - 1).\\
$$

Here we managed to give a neat classification of sorting plans of
order $2$ by looking at local neighborhoods and getting rid of the
invalid scenarios.  In the next section we will generalize this
approach, looking at what we call forbidden segments.

\section{Forbidden segments}

A semitrace fails to be a trace precisely when there is a pair of elements
witnessing Property~\ref{traceprop1} of Definition~\ref{def:trace}
fail. With the following definition we single out such pairs.

\begin{definition}\label{def:violating_pair}
  Let $T$ be a semitrace of length $n$ and let $a$ and $b$ be two
  distinct elements of $[n]$. We call $(a,b)$ a \emph{violating pair} of
  $T$ if, in any row, $a$ and $b$ occur in adjacent positions such that
  they violate Property~\ref{traceprop1}. That is, $a$ and $b$ form a
  descent and are separated by a bar, or $a$ and $b$ form an ascent and
  are not separated by a bar.
\end{definition}

From Definitions~\ref{def:trace} and~\ref{def:violating_pair} we immediately
get the following characterization of sorting plans.
\begin{lemma}\label{lem:invalid_skeleton_violating_pair}
  An operation array is a sorting plan if and only if its semitrace
  has no violating pair $(a,b)$.
\end{lemma}

\begin{definition}
  Let $\psi$ be the bijection mapping a sorting plan to its encoding.
  Let $M$ be a sorting plan of length $n$ and let $i$ and $j$ be two indices with
  $1\leq i\leq j\leq n+1$. Let $\psi(M) = c_1c_2\ldots c_{n+1}$ be the encoding of
  $M$. Then we call $S = \psi^{-1}(c_ic_{i+1}\ldots c_{j})$ a
  \emph{segment} of $M$, and $|S| = j-i+1$ its length.
\end{definition}

\begin{definition}
  Let $T=(A,M)$ be a semitrace of length $n$ and let $a$ and $b$ be two
  distinct elements of $[n]$. Let $B$ be the set of blocks that contain either
  $a$ or $b$, excluding blocks in the first row. We define \emph{the segment of $T$
  determined by $a$ and $b$}, denoted $T_{a,b}$, as the smallest segment of $M$
  that fully contains all the blocks in $B$.
\end{definition}


As an example, consider the semitrace
$$
\begin{tikzpicture}[line width=\customwidth, scale=\customscale]
  \useasboundingbox (0,-5+0.5) rectangle (8,1);
  \comparray{(0.0,0)}{
    {1,0,0,1,1,0,0,0,1},
    {1,0,1,0,1,1,1,0,1},
    {1,1,0,1,0,0,1,1,1},
    {1,0,1,0,1,1,0,1,1},
    {1,1,0,1,0,1,1,1,1}};
  \permutation{(0.0,0)}{5}{1,2,3,4,5,6,7,8};
  \permutation{(0.0,0)}{4}{1,3,2,5,4,6,7,8};
  \permutation{(0.0,0)}{3}{3,1,5,2,4,7,6,8};
  \permutation{(0.0,0)}{2}{3,5,1,7,4,2,6,8};
  \permutation{(0.0,0)}{1}{5,3,7,1,4,2,8,6};
  \permutation{(0.0,0)}{0}{7,3,5,1,6,8,2,4};
\end{tikzpicture}
$$
It contains four violating pairs, namely $(1,5)$, $(2,4)$, $(3,5)$, and
$(6,8)$. Let us look at the pair $(2,4)$ and how these two numbers
progress through the trace:
$$
\begin{tikzpicture}[line width=\customwidth, scale=\customscale]
  \useasboundingbox (0,-5+0.4) rectangle (8,1);
  \comparray{(0.0,0)}{
    {1,0,0,1,1,0,0,0,1},
    {1,0,1,0,1,1,1,0,1},
    {1,1,0,1,0,0,1,1,1},
    {1,0,1,0,1,1,0,1,1},
    {1,1,0,1,0,1,1,1,1}};
  \draw[black] (1.5,-4.5) --
               (2.5,-3.5) --
               (3.5,-2.5) --
               (5.5,-1.5) --
               (5.5,-0.5) --
               (6.5,0.5);
  \draw[black] (3.5,-4.5) --
               (4.5,-3.5) --
               (4.5,-2.5) --
               (4.5,-1.5) --
               (4.5,-0.5) --
               (7.5,0.5);
  \foreach \y/\x in {5/1,
                     4/2,
                     3/3,
                     2/5,
                     1/5,
                     0/6,
                     5/3,
                     4/4,
                     3/4,
                     2/4,
                     1/4,
                     0/7} {
      \fill[white] (\x+0.5,-\y+0.5) circle [radius=0.85em];
      \draw[black] (\x+0.5,-\y+0.5) circle [radius=0.85em];
  };
  \permutation{(0.0,0)}{5}{1,2,3,4,5,6,7,8};
  \permutation{(0.0,0)}{4}{1,3,2,5,4,6,7,8};
  \permutation{(0.0,0)}{3}{3,1,5,2,4,7,6,8};
  \permutation{(0.0,0)}{2}{3,5,1,7,4,2,6,8};
  \permutation{(0.0,0)}{1}{5,3,7,1,4,2,8,6};
  \permutation{(0.0,0)}{0}{7,3,5,1,6,8,2,4};
\end{tikzpicture}
$$
The segment $T_{2,4}$ is the smallest segment that contains all blocks
on rows $2$, $3$, $4$, and $5$ with at least one circled element:
$$
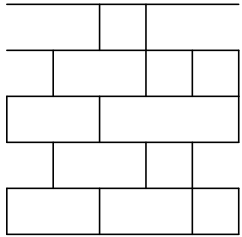
\begin{tikzpicture}[line width=\customwidth, scale=\customscale]
  \useasboundingbox (0,-4+0.2) rectangle (5,1);
  \comparray{(0.0,0)}{
    {0,0,1,1,0,0},
    {0,1,0,1,1,1},
    {1,0,1,0,0,1},
    {0,1,0,1,1,0},
    {1,0,1,0,1,1}};
\end{tikzpicture}
$$
Any sorting plan of order $5$ that contains this segment, no matter where
it occurs horizontally, will violate Property~\ref{traceprop1}, just as the
above semitrace. This is because we can follow the two elements $a$ and $b$ playing
the roles of $2$ and $4$ from the bottom to the second row.
$$
  
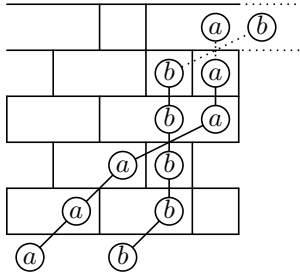
\begin{tikzpicture}[line width=\customwidth, scale=\customscale]
    \useasboundingbox (0,-4.6) rectangle (5,1);
    \comparray{(0.0,0)}{
      {0,0,1,1,0,0},
      {0,1,0,1,1,1},
      {1,0,1,0,0,1},
      {0,1,0,1,1,0},
      {1,0,1,0,1,1}};
    \draw[black] (0.5,-4.5) --
                 (1.5,-3.5) --
                 (2.5,-2.5) --
                 (4.5,-1.5) --
                 (4.5,-0.5);
    \draw[black,densely dotted] (4.5,-0.5) -- (4.5,0.5);
    \draw[black] (2.5,-4.5) --
                 (3.5,-3.5) --
                 (3.5,-2.5) --
                 (3.5,-1.5) --
                 (3.5,-0.5);
    \draw[black,dotted] (3.5,-0.5) -- (5.5,0.5);

    \draw[black,dotted] (5,1) --
                        (6.5,1) --
                        (6.5,0) --
                        (5,0);

    \foreach \y/\x in {5/1,
                       4/2,
                       3/3,
                       2/5,
                       1/5,
                       0/5,
                       5/3,
                       4/4,
                       3/4,
                       2/4,
                       1/4,
                       0/6} {
        \fill[white] (\x-1+0.5,-\y+0.5) circle [radius=0.85em];
        \draw[black] (\x-1+0.5,-\y+0.5) circle [radius=0.85em];
      };
    \permutation{(0.0,0)}{5}{a,,b,,,,};
    \permutation{(0.0,0)}{4}{,a,,b,,,};
    \permutation{(0.0,0)}{3}{,,a,b,,,};
    \permutation{(0.0,0)}{2}{,,,b,a,,};
    \permutation{(0.0,0)}{1}{,,,b,a,,};
    \node[text height=1.5ex,text depth=.25ex] at (4.5,0.5) {$a$};
    \node[text height=1.5ex,text depth=.25ex] at (5.5,0.5) {$b$};
  \end{tikzpicture}
$$
Formally, we make the following definition.
\begin{definition}\label{def:forbidden_segment}
  A segment
  $T_{a,b}$ is \emph{forbidden} if, after inferring the positions of
  $a$ and $b$ in rows $2$ to
  $k$, either the two numbers violate Property~\ref{traceprop1} on row
  $i$, where $2 \leq i \leq k$, or
  \begin{enumerate}
  \item\label{forbidden_segment_case_1} $a$ and $b$ form a descent on row $2$ and, if
    $a$ and $b$ are in columns $i$ and $i+1$ on row $2$, there is no bar
    separating columns $i$ and $i+1$ on row $1$, or
  \item\label{forbidden_segment_case_2} $a$ and
    $b$ are in decreasing order on row $2$ and, if $a$ and
    $b$ are in columns $i$ and $j$ on row $2$, with
    $i<j$, there is a bar immediately to the left of column
    $i$ on row $1$, a bar immediately to the right of column $j$ on row
    $1$, and exactly one bar between columns $i$ and $j$ on row $1$.
  \end{enumerate}
  In both cases the bars that we reference belong to the segment
  $T_{a,b}$. In other words, we can determine if
  $T_{a,b}$ is forbidden or not without knowing in which semitrace it is embedded.
\end{definition}

\begin{lemma}\label{lem:forbidden_segment_violating_pair}
  A segment $T_{a,b}$ is forbidden if and only if $(a,b)$ is a violating
  pair of $T$.
\end{lemma}
\begin{proof}
  Consider a forbidden segment $T_{a,b}$. If, after inferring the
  positions of $a$ and $b$ in rows $2$ to $k$, the two numbers violate
  Property~\ref{traceprop1} on row $i$, where $2\leq i\leq k$, then
  $(a,b)$ is a violating pair of $T$. Otherwise we have two cases,
  corresponding to the two cases of Definition~\ref{def:forbidden_segment},
  and these are illustrated in Figures~\ref{fig:forbidden_segment_1} and
  \ref{fig:forbidden_segment_2}, respectively.
  In the first case $a$ and  $b$ form an ascent and are in the same block on row $1$.
  In the second case $a$ and $b$ form a descent and are separated by a bar on row $1$.
  In both cases we have found a violation of Property~\ref{traceprop1}
  and hence $(a,b)$ is a violating pair of $T$.
  Conversely, consider a violating pair $(a,b)$ of $T$, as well as the
  segment $T_{a,b}$. If $a$ and $b$ violate Property~\ref{traceprop1} on row
  $i$, where $2\leq i\leq k$, then $T_{a,b}$ is a forbidden segment. If, on
  the other hand, the two numbers violate Property~\ref{traceprop1} on row $1$, then
  we have two cases:
  \begin{itemize}
  \item If, on row $1$, $a$ and $b$ form an ascent and are not separated
    by a bar, then the two numbers are in the same block on row $1$,
    and will form a descent on row $2$. Furthermore, if the two numbers
    are in columns $i$ and $i+1$ on row $2$, there will not be a bar
    separating columns $i$ and $i+1$ on row $1$. Hence $T_{a,b}$ is a
    forbidden segment.
  \item If, on row $1$, $a$ and $b$ form a descent and are separated by a
    bar, then the two numbers are in adjacent blocks on row $1$. If $i$
    is the leftmost column that the left block intersects, and $j$ is
    the rightmost column that the right block intersects, then the two
    numbers will be in decreasing order on row $2$, with the larger
    number number in column $i$ and the smaller number in column $j$.
    Furthermore, there is a bar immediately to the left of column $i$ on row $1$
    and a bar immediately to the right of column $j$ on row $1$, and exactly one
    bar between columns $i$ and $j$ on row $1$. Hence $T_{a,b}$ is a
    forbidden segment.
  \end{itemize}
  In both cases $T_{a,b}$ is a forbidden segment.
\end{proof}
\begin{figure}
  \centering
  \begin{tikzpicture}[line width=\customwidth, scale=\customscale]
    \useasboundingbox (-1,0.5) rectangle (13,2.5);
    \draw (0,2) rectangle (12,1);
    \begin{scope}[every node/.style={text height=1.5ex,text depth=.25ex}]
      \node at (2.5, 1.5) {$a$}; \node at ( 3.5, 1.5) {$b$};
      \node at (8.5, 0.5) {$b$}; \node at ( 9.5, 0.5) {$a$};
      \foreach \y in {1,2} {
        \node at (-0.8,-\y+2.5) {$\scriptstyle{\y}$};
      }
      \foreach \y in {1,2} { 
        \node[white] at (12.8,-\y+2.5) {$\scriptstyle{\y}$};
      }
      \node at (8.5, 2.5) {$\scriptstyle{i}$};
      \node at (9.5, 2.5) {$\scriptstyle{i+1}$};
    \end{scope}
  \end{tikzpicture}
  \caption{The first case of Definition~\ref{def:forbidden_segment}, assuming $a<b$}
  \label{fig:forbidden_segment_1}
\end{figure}
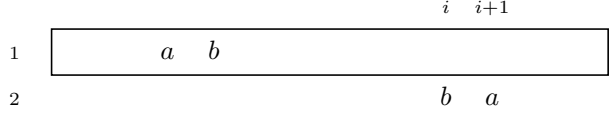
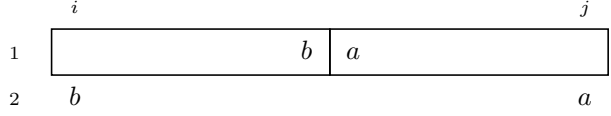
\begin{figure}
  \centering
  \begin{tikzpicture}[line width=\customwidth, scale=\customscale]
    \useasboundingbox (-1,0.5) rectangle (13,2.5);
    \draw (0,2) rectangle (6,1); \draw (6,2) rectangle (12,1);
    \begin{scope}[every node/.style={text height=1.5ex,text depth=.25ex}]
      \node at (5.5, 1.5) {$b$}; \node at (  6.5, 1.5) {$a$};
      \node at (0.5, 0.5) {$b$}; \node at ( 11.5, 0.5) {$a$};
      \foreach \y in {1,2} {
        \node at (-0.8,-\y+2.5) {$\scriptstyle{\y}$};
      }
      \foreach \y in {1,2} { 
        \node[white] at (12.8,-\y+2.5) {$\scriptstyle{\y}$};
      }
      \node at ( 0.5, 2.5) {$\scriptstyle{i}$};
      \node at (11.5, 2.5) {$\scriptstyle{j}$};
    \end{scope}
  \end{tikzpicture}
  \caption{The second case of Definition~\ref{def:forbidden_segment}, assuming $a<b$}
  \label{fig:forbidden_segment_2}
\end{figure}

From Lemmas~\ref{lem:invalid_skeleton_violating_pair} and
\ref{lem:forbidden_segment_violating_pair} we get the following lemma.
\begin{lemma}\label{lem:operation_array_skeleton}
  An operation array is a sorting plan if and only if it does not contain
  any forbidden segment $T_{a,b}$.
\end{lemma}

While this lemma is interesting, it is of limited practical use the way
it is stated. The reason being that there are potentially an infinite
number of forbidden segments $T_{a,b}$. This motivates the following
definition.

\begin{definition}
  A segment of order $k$ is \emph{bounded} if each of its blocks in rows
  $2$ to $k$ has size at most $3$. Equivalently, a segment is bounded
  if its operation array has no occurrence of three consecutive \D's on
  rows $2$ to $k$.
\end{definition}

\begin{proposition}\label{prop:operation_array_skeleton_bounded}
  An operation array is a sorting plan if and only if it does not contain any
  bounded forbidden segment $T_{a,b}$ and each block on rows $2$ through $k$ is
  of size at most $3$.
\end{proposition}
\begin{proof}
  Let a sorting plan be given.  By
  Lemma~\ref{lem:operation_array_skeleton} it contains no forbidden
  segment $T_{a,b}$, and, in particular, no bounded forbidden segment
  $T_{a,b}$. Furthermore, each block on rows $2$ through $k$ is of size
  at most $3$ by Lemma~\ref{lem:blocks3}.

  Conversely, assume that we are given an operation array that does not
  contain any bounded forbidden segment $T_{a,b}$ and that each block on
  rows $2$ through $k$ is of size at most $3$.  If the operation array
  is not a sorting plan, then the operation array contains a forbidden
  segment $T_{a,b}$ by Lemma~\ref{lem:operation_array_skeleton}. This
  forbidden segment is not bounded, so it has to contain a block of size
  greater than $3$ in one of the rows $2$ through $k$, a contradiction.
  Hence the operation array must be a sorting plan.
\end{proof}

\begin{lemma}\label{lem:forbidden_segment_bound}
    Let $T$ be a semitrace of length $n$ and order $k$ and let $a$ and $b$ be
    two distinct elements of $[n]$. If $T_{a,b}$ is a bounded segment, and there is
    a block, not on the first row, that includes both $a$ and $b$, then
    $|T_{a,b}| \leq 4k-5$.
\end{lemma}
\begin{proof}
  First note that we can completely disregard the first row, as neither this
  lemma nor the definition of $T_{a,b}$ includes blocks on that row. Since
  $T_{a,b}$ is a bounded segment, each of the remaining blocks that either
  contains $a$ or $b$ has size at most $3$. If we consider two adjacent rows,
  and $x\in\{a,b\}$, this implies that the horizontal distance between $x$ in
  the upper row and $x$ in the lower row is at most $2$. In total, the
  horizontal distance between $x$ on the second row and $x$ on the $k$-th row
  is at most $2(k-2)$ as illustrated in Figure~\ref{fig:wide-skeleton}. From
  this we see that the length of the segment $T_{a,b}$ is at most $4(k-2)+m$,
  where $m$ is the size of the block that contains $a$ and $b$. Noting that
  $m\leq 3$ gives us the desired bound.
  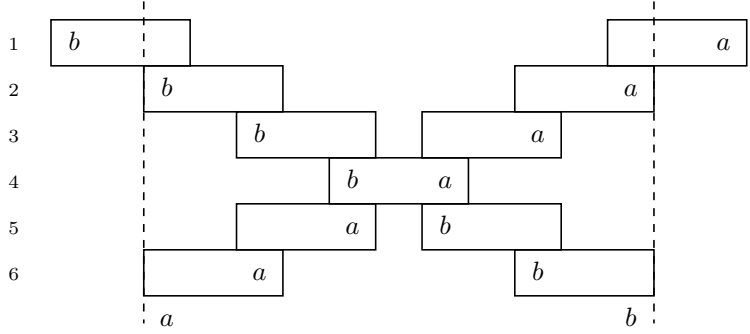
\begin{figure}\label{fig:wide-skeleton}
    \centering
    \begin{tikzpicture}[line width=\customwidth, scale=\customscale]
      \useasboundingbox (-3,-4.5) rectangle (14,2.3);
      \draw [dashed] (0,2.4) -- (0,-4.6);
      \draw [dashed] (11,2.4) -- (11,-4.6);
      \draw (-2, 2) rectangle (1, 1); \draw (10, 2) rectangle (13, 1);
      \draw (0, 1) rectangle (3, 0); \draw (8, 1) rectangle (11, 0);
      \draw (2, 0) rectangle (5,-1); \draw (6, 0) rectangle ( 9,-1);
      \draw (4,-1) rectangle (7,-2);
      \draw (2,-2) rectangle (5,-3); \draw (6,-2) rectangle ( 9,-3);
      \draw (0,-3) rectangle (3,-4); \draw (8,-3) rectangle (11,-4);
      \begin{scope}[every node/.style={text height=1.5ex,text depth=.25ex}]
        \node at (-2+0.5, 1+0.5) {$b$}; \node at (12+0.5, 1+0.5) {$a$};
        \node at (0+0.5, 0+0.5) {$b$}; \node at (10+0.5, 0+0.5) {$a$};
        \node at (2+0.5,-1+0.5) {$b$}; \node at ( 8+0.5,-1+0.5) {$a$};
        \node at (4+0.5,-2+0.5) {$b$}; \node at ( 6+0.5,-2+0.5) {$a$};
        \node at (4+0.5,-3+0.5) {$a$}; \node at ( 6+0.5,-3+0.5) {$b$};
        \node at (2+0.5,-4+0.5) {$a$}; \node at ( 8+0.5,-4+0.5) {$b$};
        \node at (0+0.5,-5+0.5) {$a$}; \node at ( 10+0.5,-5+0.5) {$b$};
        \foreach \y in {1,2,3,4,5,6} {
          \node at (-2.8,-\y+2.5) {$\scriptstyle{\y}$};
        }
        \foreach \y in {1,2,3,4,5,6} { 
          \node[white] at (13.8,-\y+2.5) {$\scriptstyle{\y}$};
        }
      \end{scope}
    \end{tikzpicture}
    \caption{The blocks containing $a$ or $b$ in a trace of order 6 with
      the dashed lines indicating the boundary of the segment $T_{a,b}$}
  \end{figure}
\end{proof}

\begin{lemma}\label{lem:block_with_ab}
    Let $T$ be a semitrace and $(a,b)$ a violating pair of $T$. Then there
    is a block, not on the first row, that includes both $a$ and $b$.
\end{lemma}
\begin{proof}
    Consider a row $i$ where the pair $(a,b)$ violates
    Property~\ref{traceprop1}. We have the following two cases:
    \begin{itemize}
        \item If $a$ and $b$ form a descent and are separated by a bar, then
            the two elements are in distinct blocks on row $i$, and will
            still be in descending order on row $i+1$ after performing the
            blockwise reversals.
        \item If $a$ and $b$ form an ascent and are not separated by a bar,
            then the two elements are in the same block on row $i$, and will
            be in descending order on row $i+1$ after performing the
            blockwise reversals.
    \end{itemize}
    In either case, $a$ and $b$ will be in descending order on row $i+1$. Since
    the two elements are in increasing order in the last permutation, i.e.\ the
    identity permutation, the two elements must be reversed on at least
    one of the rows between $i+1$ and $k$. Since the relative order of elements
    is only reversed when they appear together in a block, there must be a
    block on one of the rows between $i+1$ and $k$ that includes both $a$ and
    $b$.
\end{proof}

\begin{lemma}\label{lem:finite_forbidden_segments}
  For a fixed $k$, there are finitely many bounded forbidden segments $T_{a,b}$
  of order $k$, and they can be listed.
\end{lemma}
\begin{proof}
  Assume that $T_{a,b}$ is a bounded forbidden segment. Then $(a,b)$ is a
  violating pair by
  Lemma~\ref{lem:forbidden_segment_violating_pair}. Moreover, using
  Lemma~\ref{lem:block_with_ab}, we find that there is a block, not on
  the first row, that includes both $a$ and $b$.
  Lemma~\ref{lem:forbidden_segment_bound} thus applies and
  $|T_{a,b}|\leq 4k-5$. Viewing the segment $T_{a,b}$ as an
  operation array, there consequently are at most as many forbidden segments
  $T_{a,b}$ as there are words in the finite language
  $\{\A,\D\}^{(4k-5)k}$, and they can be listed by checking each
  such operation array against Definition~\ref{def:forbidden_segment}.
\end{proof}

\section{Regular language}

Now that we have a characterization of sorting plans in terms of forbidden
segments, we will use that characterization to count sorting plans of order $k$,
and hence the $k$-pop-stack-sortable permutations. To do so, we will employ the
theory of formal languages.

Recall that we can encode an operation array of length $n$ and order $k$ as a
sequence of $n$ integers, each in the range $[0,2^k-1]$.
In this way we can consider operation arrays as strings of a formal language over the
alphabet $\Sigma = \{0,1,\ldots,2^k-1\}$. Conversely, strings over this
alphabet can be considered as operation arrays, under one condition: that they both
begin and end with a solid boundary. Noting that a solid boundary corresponds
to the integer $0$ from $\Sigma$, and letting $\overline{0} =
\Sigma\setminus\{0\}$, the following deterministic finite automaton (DFA), $W$,
recognizes the strings over $\Sigma$ that begin and end with a solid boundary,
i.e.\ the strings that correspond to an operation array:
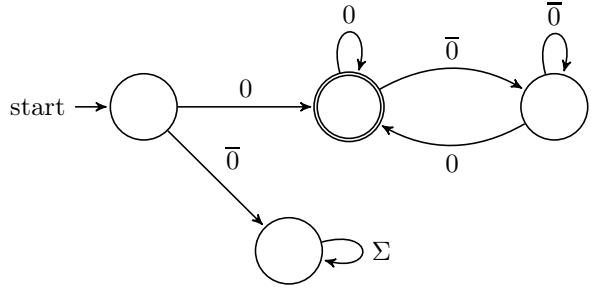
\begin{center}
\begin{tikzpicture}[->,>=stealth',shorten >=1pt,auto,node distance=2.7cm,
                    semithick]
  \node[initial,state] (A) {};
  \node[accepting,state] (B) [right of=A] {};
  \node[state] (C) [right of=B] {};
  \node[state] (D) [below right of=A] {};

  \path (A) edge              node {$0$} (B)
        (B) edge[bend left] node {$\overline{0}$} (C)
        (C) edge[bend left] node {$0$} (B)
        (B) edge[loop above] node {$0$} (B)
        (C) edge[loop above] node {$\overline{0}$} (C)
        (D) edge[loop right] node {$\Sigma$} (D)
        (A) edge node {$\overline{0}$} (D);
\end{tikzpicture}
\end{center}

For ease of notation we will use the name of a DFA to also denote the
language that it recognizes. Here, $W$ denotes the DFA recognizing
strings that begin and end with a solid boundary as well as the language
consisting of such strings. We want to find the subset of $W$
corresponding to sorting plans. Recall from
Proposition~\ref{prop:operation_array_skeleton_bounded} that an operation array is
a sorting plan if and only if it does not contain any bounded forbidden segments and
each block on rows $2$ through $k$ is of size at most $3$. We shall start
with the latter condition.

If $A_i$ is the set of symbols from $\Sigma$ that represent a column
from the operation array that has a bar in the $i$th row, and $\overline{A_i} =
\Sigma\setminus A_i$, then the following DFA, $R_i$, recognizes the operation
arrays that have blocks of size at most $3$ in row $i$:
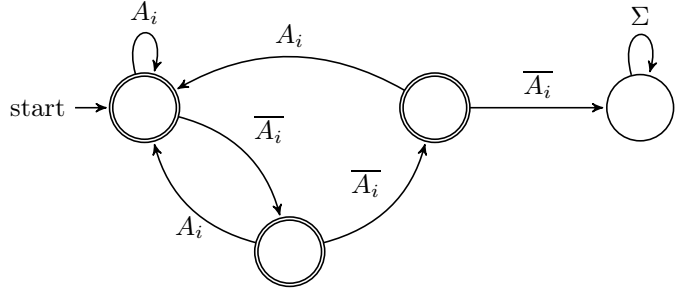
\begin{center}
\begin{tikzpicture}[->,>=stealth',shorten >=1pt,auto,node distance=2.7cm,
                    semithick]
  \node[initial,accepting,state] (A) {};
  \node[accepting,state] (B) [below right of=A] {};
  \node[accepting,state] (C) [above right of=B] {};
  \node[state] (D) [right of=C] {};

    \path (A) edge[bend left] node {$\overline{A_i}$} (B)
          (B) edge[bend right] node {$\overline{A_i}$} (C)
          (C) edge node {$\overline{A_i}$} (D)
          (D) edge[loop above] node {$\Sigma$} (D)
          (A) edge[loop above] node {$A_i$} (A)
          (B) edge[bend left] node[below] {$A_i$} (A)
          (C) edge[bend right] node[above] {$A_i$} (A);
\end{tikzpicture}
\end{center}
Therefore, the set of operation arrays that have blocks of size at most $3$ in
all but the first row is recognized by the DFA $W\cap R_2 \cap \cdots \cap R_k$.

The other condition that sorting plans satisfy is that they do not contain any
bounded forbidden segments. Consider a segment $M$ and let us encode it in the
same manner as we encode operation arrays, resulting in the sequence
$m_1,\ldots,m_{\ell}$. Note that an operation array $A$ contains the segment $M$
if and only if the encoding of $A$ contains $m_1\cdots m_{\ell}$ as a factor.
Furthermore, the following nondeterministic finite automaton (NFA), $Q_M$, recognizes
the set of strings over $\Sigma$ that contain the encoding of $M$ as a
factor:
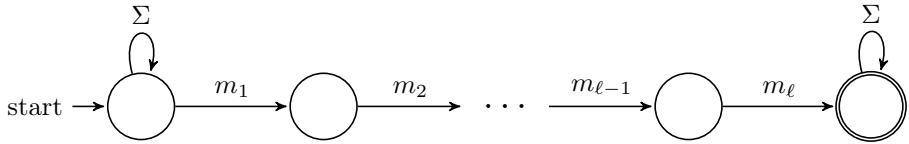
\begin{center}
\begin{tikzpicture}[->,>=stealth',shorten >=1pt,auto,node distance=2.4cm,
                    semithick]
    \node[initial,state] (A) {};
    \node[state] (B) [right of=A] {};
    \node (C) [right of=B] {$\scalebox{1.4}{\;\ldots\,}$};
    \node[state] (D) [right of=C] {};
    \node[accepting,state] (E) [right of=D] {};
    \path (A) edge node {$m_1$} (B)
          (B) edge node {$m_2$} (C)
          (C) edge node {$m_{\ell-1}$} (D)
          (D) edge node {$m_{\ell}$} (E)
          (A) edge[loop above] node {$\Sigma$} (A)
          (E) edge[loop above] node {$\Sigma$} (E);
\end{tikzpicture}
\end{center}
Taking the complement of $Q_M$ we get an automaton $\overline{Q_M}$ that
recognizes the set of strings over $\Sigma$ that do not contain the factor
$M$. In particular, if $F$ is a forbidden segment, then $W\cap \overline{Q_F}$
recognizes the set of operation arrays that do not contain the forbidden
segment $F$.

Let $\mathcal{F}$ be the set of bounded forbidden segments, which is finite by
Lemma~\ref{lem:finite_forbidden_segments}. Then the automaton
\[
    S = W\cap \bigcap_{i=2}^k R_i \cap \bigcap_{F\in\mathcal{F}} \overline{Q_F}
\]
recognizes the set of operation arrays that have blocks of size at most $3$ in
rows $2$ through $k$, and do not contain any bounded forbidden segments. Hence, by
Proposition~\ref{prop:operation_array_skeleton_bounded}, the automaton $S$ recognizes exactly
the set of sorting plans. We have the following proposition:
\begin{proposition}\label{prop:skel_regular}
    The language $S = \{\ w\in\Sigma^{*}\ |\ w\text{ is a sorting plan}\ \}$ is regular.
\end{proposition}

We can now present our main theorem.
\begin{theorem}
  For a fixed $k$, the generating function
  $P(x) = \sum_{n=0}^{\infty}p_{n} x^{n}$, where $p_{n}$
  is the number of $k$-pop-stack-sortable permutations of length $n$, is
  rational.
\end{theorem}
\begin{proof}
    We have a bijection between the $k$-pop-stack-sortable permutations of
    length $n$ and the sorting plans of order $k$ and length $n$, so the two sets
    are equinumerous. The sorting plans of length $n$ are in bijection with words
    of length $n+1$ recognized by the automaton $S$, which is regular by
    Proposition~\ref{prop:skel_regular}. It is well known that regular languages have
    rational generating functions~\cite{schutzenberger1961definition}, and
    they can be derived from the corresponding DFA by setting up a system of
    linear equations. If $S(x)$ is the rational generating function for
    $S$, it is clear that $P(x) = S(x)/x$ and that this generating function
    is rational.
\end{proof}

Since all of the above results are constructive, it is possible to derive the
generating function for any fixed $k$. Doing so by hand is, however, impractical
for almost all values of $k$, so we implemented the above constructions in the
programming languages C++ and Python; the source is on
GitHub~\cite{popstacksgithub}. This way the generating function can be
derived by a computer.  Without going into detail, the outline of the
mechanized procedure, for a fixed $k$, was as follows:
\begin{enumerate}
    \item A smart variant of the procedure given in
        Lemma~\ref{lem:finite_forbidden_segments} was used to generate a
        compact representation of all the bounded forbidden segments. This was achieved
        by grouping together similar segments, and using
        ``wildcards'' to represent a position that could either be a bar or
        not. This was done to battle the exponential blow up in the number of
        segments.
    \item For each group of bounded forbidden segments, an NFA recognizing the
        operation arrays containing one or more of the segments from the group
        was constructed. Each NFA was then turned into a DFA using the classic
        subset construction, complemented, and then minimized using an
        algorithm of Valmari~\cite{valmari2012fast}.
    \item Running a MapReduce-like~\cite{dean2008mapreduce} procedure on a
        cluster~\cite{garpur}, these DFAs were intersected, two at a
        time. To keep the size of the intermediate DFAs down, they
        were also minimized.
    \item A system of linear equations was derived from the final DFA, and the
        system was solved to get the desired generating function.
\end{enumerate}

\begin{table}
    \centering
    \newcolumntype{L}{>{\centering\arraybackslash}m{0.91\textwidth}}
    \begin{tabular}{cL}
        $k$ & Generating function \\[1em]

        $1$ & $(x - 1)/(2x - 1)$ \\[1em]
        $2$ & $(x^{3} + x^{2} + x - 1)/(2x^{3} + x^{2} + 2x - 1)$ \\[1em]
        $3$ & $(2x^{10} + 4x^{9} + 2x^{8} + 5x^{7} + 11x^{6} + 8x^{5} + 6x^{4} + 6x^{3} + 2x^{2} + x - 1)/(4x^{10} + 8x^{9} + 4x^{8} + 10x^{7} + 22x^{6} + 16x^{5} + 8x^{4} + 6x^{3} + 2x^{2} + 2x - 1)$ \\[1em]
        $4$ & $(64x^{25} + 448x^{24} + 1184x^{23} + 1784x^{22} + 2028x^{21} + 1948x^{20} + 1080x^{19} + 104x^{18} - 180x^{17} + 540x^{16} + 1156x^{15} + 696x^{14} + 252x^{13} + 238x^{12} + 188x^{11} + 502x^{10} + 806x^{9} + 544x^{8} + 263x^{7} + 185x^{6} + 99x^{5} + 33x^{4} + 13x^{3} + 3x^{2} + x - 1)/(128x^{25} + 896x^{24} + 2368x^{23} + 3568x^{22} + 3928x^{21} + 3064x^{20} + 176x^{19} - 2304x^{18} - 2664x^{17} - 1580x^{16} - 352x^{15} - 576x^{14} - 1104x^{13} - 760x^{12} - 138x^{11} + 686x^{10} + 1238x^{9} + 869x^{8} + 382x^{7} + 210x^{6} + 102x^{5} + 27x^{4} + 12x^{3} + 3x^{2} + 2x - 1)$
    \end{tabular}
    \caption{The generating functions for the $k$-pop-stack-sortable permutations, $k\leq 4$}
    \label{tbl:generating_functions}
\end{table}
We did so for $k=1,\ldots,6$ and in
Table~\ref{tbl:generating_functions} we list the resulting generating
functions, except for $k=5$ and $k=6$ whose expressions are too large to
display.  For each of them the degree of the polynomial in the numerator
is the same as the degree of the polynomial in the denominator.
Those degrees, the growth rates of coefficents of the generating functions,
and the corresponding sequences for the
number of vertices and edges in the final DFAs can be found in the table
below.
$$
\begin{array}{c|cccccc}
  k & 1 & 2 & 3 & 4 & 5 & 6 \\\hline
  \text{degree} &1&3&10&25&71&213 \\
  \text{growth rate} & 2.0000 & 2.6590 & 3.4465 & 4.2706 & 5.1166 & 5.9669 \\
  \text{vertices}&4&5&12&32&99&339 \\
  \text{edges}&8&11&34&120&477&2010
\end{array}
$$
All the generating functions, source code, and text files defining the
DFAs can be found on GitHub~\cite{popstacksgithub}. The growth rate of
the coefficients of a rational power series $p(x)/q(x)$ is given by
$\max\{1/|\zeta|: q(\zeta)=0\}$ and Sage~\cite{sage} code for
calculating the approximate growth rates, in the table above, can be found
on the same GitHub page.

While we do not think it would be particularly
interesting to compute generating functions for higher $k$ using our
algorithm, it would be interesting to find a closed formula for the
generating functions, possibly leading to results about the
distribution of the number of passes needed to sort a permutation using
a pop-stack. It is not clear whether our approach can be used as a basis
for such a formula.

When deriving the generating functions, we observed that, during the
phase when the DFAs are intersected, the intermediate DFAs were quite
large, often consisting of millions of states. As the final DFAs are so
small, this may indicate that our approach is not a natural one, and
that simpler, more direct approaches exist. And with a simpler approach
it may be easier to find a closed formula.
Finally, as we only considered their enumeration, finding a useful
permutation pattern characterization of the $k$-pop-stack-sortable
permutations remains open.

\bibliographystyle{plain}
\bibliography{references}

\end{document}